\documentclass{amsart}
 \newtheorem{thm}{Theorem}[section]
 \newtheorem{cor}[thm]{Corollary}
 
 \newtheorem{prop}[thm]{Proposition}
 \theoremstyle{definition}
 \newtheorem{defn}[thm]{Definition}
 \theoremstyle{remark}

 \numberwithin{equation}{section}
 
\usepackage{amscd,bbm}
\usepackage{multicol}

\newcommand{\alp}{\alpha}

\newcommand{\la}{\lambda}
\newcommand{\eps}{\varepsilon}

\newcommand{\noi}{\noindent}
\newcommand{\rr}{\mathbb{R}}

\newcommand{\nn}{\mathbb{N}}
\newcommand{\zz}{\mathbb{Z}}

\newcommand{\bm}{\begin{pmatrix}}
\newcommand{\fm}{\end{pmatrix}}
\newcommand{\und}{\underset}

\newcommand{\inte}{\cap}

\newcommand{\dd}{\ {\rm d}}

\newcommand{\inc}{\subset}

\newcommand{\ieg}{\left[\kern-0.15em\left[ \,}
\newcommand{\ied}{\, \right]\kern-0.15em\right]}

\newcommand{\dis}{\displaystyle}
\newcommand{\hs}{\hspace{0.02cm}}

\newcommand{\ft}{\widetilde{f}}
\newcommand{\gat}{\widetilde{A}}
\newcommand{\at}{\widetilde{a}}
\newcommand{\Br}{\text{Br}}

\newcommand{\lun}{\ell_{1}}

\newcommand{\Lun}{L_{1}}
\newcommand{\Linf}{L_{\infty}}

\newcommand{\fz}{\mathcal{F}}
\newcommand{\lipz}{\text{Lip}_0}

\newcommand{\lipfss}{Lipschitz-free spaces }

\newcounter{cnt1}
\newcounter{cnt2}
\newcounter{cnt3}
\newcommand{\blr}{\begin{list}{$($\roman{cnt1}$)$}
 {\usecounter{cnt1} \setlength{\topsep}{0pt}
 \setlength{\itemsep}{0pt}}}
\newcommand{\bla}{\begin{list}{$($\alph{cnt2}$)$}
 {\usecounter{cnt2} \setlength{\topsep}{0pt}
 \setlength{\itemsep}{0pt}}}
\newcommand{\bln}{\begin{list}{$($\arabic{cnt3}$)$}
 {\usecounter{cnt3} \setlength{\topsep}{0pt}
 \setlength{\itemsep}{0pt}}}
\newcommand{\el}{\end{list}}

\begin{document}
%
%
%
%
\title[Tree metrics and their Lipschitz-free spaces]
 {Tree metrics and their Lipschitz-free spaces}
\author{A. Godard}
\address{
Universit\'e Paris 6\\
IMJ - Projet Analyse Fonctionnelle\\
Bo\^ite 186\\
4 place Jussieu\\
75252 Paris C\' edex 05\\
France
}
\email{godard@math.jussieu.fr}

\subjclass[2000]{Primary 46B04; Secondary 05C05, 46B25, 54E35}

\keywords{Lipschitz-free spaces, subspaces of $L_{1}$, metric trees, four-point property.}

\date{\today}

\begin{abstract}
We compute the Lipschitz-free spaces of subsets of the real line and characterize subsets of metric trees by the fact that their Lipschitz-free space is isometric to a subspace of $L_1$.
\end{abstract}

\maketitle
\section{Introduction}
Let $M$ be a pointed metric space (a metric space with a designated origin denoted $0$). We denote by $\lipz(M)$ the space of all real-valued Lipschitz functions on $M$ which vanish at $0$ equipped with the standard Lipschitz norm
$$\|f\|=\inf \ \{K \in \rr \  / \ \forall \, x,y \in M \  \  |f(y)-f(x)| \leq K \, d(x,y) \} \quad .$$
The closed unit ball of this space being compact for the topology of pointwise convergence on $M$, $\lipz(M)$ has a canonical predual, namely the closed linear span of the evaluation functionals $\delta_x$ (whether this predual is unique up to a linear isometry is an open question). We call Lipschitz-free space over $M$ and denote by $\fz(M)$ this predual. Lipschitz-free spaces are studied extensively in \cite{Wea1999} where they are called Arens-Eels spaces. Note that for every point $a$ of $M$, $f \mapsto f-f(a)$ is a weak*-weak* continuous linear isometry between $\lipz(M)$ and $\text{Lip}_a(M)$: the choice of different base points therefore yields isometric Lipschitz-free spaces.
\\
The key property of the Lipschitz-free space $\fz(M)$ is that a Lipschitz map between two metric spaces admits a linearization between the corresponding Lipschitz-free spaces. More precisely, if $M_1$ and $M_2$ are pointed metric spaces and $L : M_1 \to M_2$ is a Lipschitz map satisfying $L(0)=0$, there exists a unique linear map $\widehat{L} : \fz(M_1) \to \fz(M_2)$ such that the following diagram commutes:
\begin{equation*}
\begin{CD}
M_1 @>L>> M_2 \\
@V{\delta}VV @VV{\delta}V \\
\fz(M_1) @>\widehat{L}>> \fz(M_2) 
\end{CD}
\end{equation*}
(we refer to \cite{GodKal2003} for details on this functorial property). In particular, 
$\fz(M')$ is an isometric subspace of $\fz(M)$ whenever $M'$ isometrically embeds into $M$.\\
Differentiation almost everywhere yields a weak*-weak* continuous linear isometry between $\lipz(\rr)$ and $\Linf$ which predualizes into a linear isometry between $\fz(\rr)$ and $L_1$ (the discrete version of this argument provides $\fz(\nn) \equiv \ell_1$).\\
According to a result of G. Godefroy and M. Talagrand (\cite{GodTal1981}), any Banach space isomorphic to a subspace of $\Lun$ is unique isometric predual: given the open question mentioned above, it is therefore natural to highlight metric spaces which admit a subspace of $L_1$ as their Lipschitz-free space.\\
The purpose of this article is to show that $\fz(M)$ is linearly isometric to a subspace of $L_1$ if and only if $M$ isometrically embeds into an $\rr$-tree (i.e. a metric space where any couple of points is connected by a unique arc isometric to a compact interval of the real line).\\
This gives a complement to a result of A. Naor and G. Schechtman that $\fz(\rr^2)$ is not isomorphic to a subspace of $L_1$ (\cite{NaoSch2007}).\\
We also compute the Lipschitz-free spaces of subsets of the real line by integrating bounded measurable functions with respect to a measure adapted to the subset considered.

\section{Preliminaries}

In this section we define differentiation for functions defined on a subset of a pointed $\rr$-tree, and introduce measures which enable to retrieve the values of such absolutely continuous functions by integrating their derivative.\\
We first recall the definition of $\rr$-trees, and define the analogous of the Lebesgue measure for those metric spaces (this measure is called the length measure).

\begin{defn}
\label{defrt}
An $\rr$-tree is a metric space $T$ satisfying the following two conditions.\\
\textup{(1)} For any points $a$ and $b$ in $T$, there exists a unique isometry $\phi$ of the closed interval $[\, 0 \, , d(a,b) \, ]$ into $T$ such that $\phi(0)=a$ and $\phi(d(a,b))=b$.\\
\textup{(2)} Any one-to-one continuous mapping $\varphi : [\, 0 \, , 1 \, ] \to T$ has same range as the isometry $\phi$ associated to the points $a=\varphi(0)$ and $b=\varphi(1)$.
\end{defn}


If $T$ is an $\rr$-tree, we denote for any $x$ and $y$ in $T$, $\phi_{xy}$ the unique isometry associated to $x$ and $y$ as in definition \ref{defrt} and write $[ \, x \, , y \, ]$ for the range of $\phi_{xy}$; such subsets of $T$ are called segments. We say that a subset $A$ of $T$ is measurable whenever $\phi_{xy}^{-1}(A)$ is Lebesgue-measurable for any $x$ and $y$ in $T$. If $A$ is measurable and $S$ is a segment $[ \, x \, , y \, ]$, we write $\lambda_S(A)$ for $\lambda \left(\phi_{xy}^{-1}(A) \right)$ where $\lambda$ is the Lebesgue measure on $\rr$. We denote $\mathcal{R}$ the set of all subsets of $T$ which can be written as a finite union of disjoints segments and for $R=\bigcup_{k=1}^{n} S_k$ (with disjoints $S_k$) in $\mathcal{R}$, we put
$$\la_R(A)=\sum_{k=1}^n \lambda_{S_k}(A) \quad .
$$

\noi Now,
$$\la_T(A)=\sup_{R \in \mathcal{R}} \la_R(A)$$
defines a measure on the $\sigma$-algebra of $T$-measurable sets such that
\begin{equation}
\label{chvar}
\int_{[ x , y ]} f(u) \dd \la_T(u)
=
\int_{0}^{d(x , y)} f(\phi_{xy}(t)) \dd t
\end{equation}
for any $f \in \Lun(T)$ and $x$, $y$ in $T$.\\

\begin{defn}
\label{defmu}
Let $T$ be a pointed $\rr$-tree, and $A$ a closed subset of $T$. We denote $\mu_A$ the positive measure defined by
$$
\mu_A=\la_A + \sum_{a \in A} L(a) \hs \delta_{a}
$$
where $\la_A$ is the restriction of the length measure on $A$, $L(a)=\inf_{x \in A \inte [ 0 \hs , \hs a [} d(a \, , x)$ and $\delta_{a}$ is the Dirac measure on $a$.
\end{defn}
This measure takes into account the gaps in $A$ by shifting their mass to the next available point (away from the root). When $L(a)>0$, we say that $a$ is root-isolated in $A$, and we denote $\gat$ the set of root-isolated points in $A$.

\begin{prop}
\label{dual}
If $T$ is a pointed $\rr$-tree and $A$ is a closed subset of $T$, $L_\infty (\mu_A)$ is isometric to the dual space of $L_1(\mu_A)$.
\end{prop}

When $T$ is separable, this proposition is a direct consequence of a classical theorem.

\begin{proof}
We have $L_1(\mu_A) \equiv L_1(\la_A) \oplus_1 \ell_1(\gat)$ and $L_\infty(\mu_A) \equiv L_\infty(\la_A) \oplus_\infty \ell_\infty(\gat)$: all we need to prove is $L_\infty(\la_A) \equiv {L_1(\la_A)}^{*}$. We show that the canonical embedding of $L_\infty (\la_A)$ into $L_1(\la_A)^{*}$ is onto. Let $\varphi$ be a bounded linear functional on $L_1(\la_A)$: for any $x$ and $y$ in $T$, $\varphi$ defines by restriction a bounded linear functional on $L_1(A \inte [ \, x \, , y \, ])$, therefore there exists $h_{xy}$ in $L_\infty (A \inte [ \, x \, , y \, ])$ such that
$$
\varphi(f)=\int_{T} f \hs h_{xy} \dd \la_A
$$
for any $f$ in $L_1(A \inte [ \, x \, , y \, ])$. Any two functions $h_{xy}$ and $h_{x'y'}$ coincide almost everywhere on $A \inte [ \, x \, , y \, ] \inte [ \, x' \, , y' \, ]$, and there exists $h$ in $L_\infty(\la_A)$ such that
$$
\varphi(f)=\int_{T} f \hs h \dd \la_T
$$
for any $f \in \Lun(\la_A)$ with support included in a finite union of intervals. Now such functions form a dense subspace of $L_1(\la_A)$ and the equality above is therefore valid for any $f$ in $L_1(\la_A)$.
\end{proof}

\noi We define differentiation on a pointed $\rr$-tree (also called a rooted real tree) in the following way.

\begin{defn}
\label{defdif}
Let $T$ be a pointed $\rr$-tree, $A \inc T$ and $f : A \to \rr$. If $a \in A$, we denote $\at$ the unique point in $[\, 0 \, , a \, ]$ satisfying $d(a \, , \at)=L(a)$. When
$$
\lim_{\substack{x \to \at \\ x \hs \in \hs [ 0 \hs , \hs a [}} \dfrac{f(a)-f(x)}{d(x \hs , a)}
$$
exists, we say that $f$ is differentiable at $a$, and we put $f'(a)$ to be the value of this limit.
\end{defn}

When $A$ is closed, a mapping  $f : A \to \rr$ is always differentiable on root-isolated points of $A$, $f'(a)$ being equal in this case to $(f(a)-f(\at))/L(a)$. When $A=T$, the limit corresponds to the left-derivative of $f_{a}=f \circ \phi_{a}$ at $d(0 \hs , a)$, where $\phi_{a}$ is the isometry associated to points $0$ and $a$ as in definition \ref{defrt}; if $f$ is Lipschitz on $T$, the function $f_{x}=f \circ \phi_{x}$ is differentiable almost everywhere on $[\, 0 \, , d(0,x) \, ]$ for any $x$ in $T$, therefore $f$ is differentiable almost everywhere on $T$. Moreover, the change of variable formula \ref{chvar} yields
$$
f(x)-f(0)=\int_{[ 0 , x ]} f' \dd \la_T
\quad .
$$

\section{Lipschitz-free spaces of metric trees}
In this section we compute the Lipschitz-free spaces of a certain class of subsets of $\rr$-trees. For this, we integrate bounded measurable functions defined on a closed subset $A$ with respect to the measure $\mu_A$ introduced in definition \ref{defmu}. We derive the computation of Lipschitz-free spaces of subsets of the real line.

\begin{defn}
\label{defbrp}
A point $t$ of an $\rr$-tree $T$ is said to be a branching point of $T$ if $T \setminus \{t\}$ possesses at least three connected components. We denote $\Br(T)$ the set of branching points of $T$.
\end{defn}

\begin{thm}
\label{lipa}
Let $T$ be an $\rr$-tree and $A$ be a subset of $T$ such that $\Br(T) \inc \bar{A}$. Then $\fz(A)$ is isometric to $\Lun(\mu_{\bar{A}})$.
\end{thm}

\begin{proof}
As $\fz(A)$ is isometric to $\fz(\bar{A})$, we may assume that $A$ is closed. We choose a point in $A$ as origin for $T$ and we prove that $\lipz(A)$ is isometric to $L_\infty (\mu_A)$. For this, we define a linear map $\Phi$ of $L_\infty (\mu_A)$ into $\lipz(A)$ by putting
$$
\Phi(g)(a)=
\int_{[ 0 , a ]} g \dd \mu_A
$$
The norm of a Lipschitz function $f$ defined on $A$ may be computed on intervals containing $0$: if $a$ and $b$ are in $A$, there exists $x \in T$ such that $[\, 0 \, , a \, ] \cap [\, 0 \, , b \, ]=[\, 0 \, ,  x \, ]$; as $A$ contains all the branching points of $T$, we know that $x \in A$. One of the quantities $|f(a)-f(x)| / d(x \, , a)$ and $|f(b)-f(x)| / d(x \, , b)$ must be greater or equal to $|f(b)-f(a)| / d(a \, , b)$ as $x$ belongs to the interval $[\, a \, , b \, ]$. We deduce that $\Phi$ is an isometry.\\
We now prove that $\Phi$ is onto. Let $f$ be in $\text{Lip}_0(A)$; we extend $f$ to a Lipschitz mapping $\ft$ defined on $T$ and affine on each interval contained in the complement of $A$. The function $\ft$ is differentiable almost everywhere on $T$ (in the sense of definition \ref{defdif}), therefore $f$ is differentiable $\mu_A$-almost everywhere. For any $a \in A$ we have
$$
f(a)=\int_{[0,a]} {\ft}^{\, '} \dd \la_T
=\int_{[0 , a]} f' \dd \la_A+\int_{[0 , a] \setminus A} {\ft}^{\, '} \dd \la_T
\quad .
$$
The set $B_a=[ \, 0 \, , a \, ] \setminus A$ is a union of open intervals:  
$$B_a=\bigcup_{i \in I} \ ] \, \at_i \, , a_i \, [$$
so
$$
\int_{B_a} {\ft}^{\, '} \dd \la_T=\sum_{i \in I} \, \dfrac{f(a_i)-f(\at_i)}{d(\at_i \hs , a_i)} \times L(a_i)
=\sum_{i \in I} L(a_i) f'(a_i)
=\int_{[0 , a]} f' \dd \nu_A
$$
where $\nu_A= \sum_{u \in A} L(u) \hs \delta_{u}$.\\
Hence
$$
f(a)=\int_{[0 , a]} f' \dd \mu_A
$$
for any $a$ in $A$ and we get $f=\Phi(f')$.\\
Finally, we know from proposition \ref{dual} that $L_\infty (\mu_A)$ is isometrically isomorphic to the dual space of $\Lun(\mu_{A})$. The linear isometry $\Phi$ being weak*-weak* continuous, it is the transpose of a linear isometry between $\fz(A)$ and $\Lun(\mu_A)$.
\end{proof}

\begin{cor}
\label{lipt}
If $T$ is an $\rr$-tree, $\fz(T)$ is isometric to $\Lun(T)$.
\end{cor}

The isometry associates to each function $h$ of $\Lun(T)$ the element $\alp$ of the Lipschitz-free space defined by $\alp(f)=\int_T f' h \dd \la_T$ i.e. the opposite of the derivative of $h$ in a distribution sense. Conversely, if $\alp$ is a measure with finite support, the corresponding $\Lun$-function is defined by $g(t)=\alp(C_t)$ where $C_t=\{x \in T \ / \ t \in [ \, 0 \, , x \,] \}$.\\
For a separable $\rr$-tree, we get $\fz(T) \equiv \Lun$. This shows that metric spaces having isometric Lipschitz-free spaces need not be homeomorphic (see \cite{DutFer2006} for examples of non-isomorphic Banach spaces having isomorphic Lipschitz-free spaces).

\begin{cor}
\label{lipas}
Let $T$ be a separable $\rr$-tree, and $A$ an infinite subset of $T$ such that $\Br(T) \inc \bar{A}$. If $\bar{A}$ has length measure $0$, then $\fz(A)$ is isometric to $\ell_1$. If $\bar{A}$ has positive length measure, $\fz(A)$ is isomorphic to $\Lun$.
\end{cor}

\begin{proof}
As $T$ is separable, the set of root-isolated points of $\bar{A}$ is at most countable. When $\la(\bar{A})>0$, $\fz(A)$ is isometric to one of the spaces  $\Lun$, $\dis \Lun \oplus_{1}  {\ell}_1^{n}$, $\Lun \oplus_{1}  \ell_1$:
the fact that all these spaces are isomorphic to $\Lun$ results from the Pe{\l}czy{\'n}ski decomposition method.
\end{proof}

This result allows to compute the Lipschitz-free space of any subset of $\rr$. As an application, we get that infinite discrete subsets of $\rr$ have a Lipschitz-free space isometric to $\ell_1$. Also, $\fz(K_3) \equiv \ell_1$ where $K_3$ is the usual Cantor set; the Lipschitz-free space of a Cantor set of positive measure is isometric to $\Lun \oplus_{1}  \ell_1$: this provides examples of homeomorphic metric spaces having non-isomorphic Lipschitz-free spaces.\\

Finite subsets of $\rr$-trees which contain the branching points usually appear as weighted trees. 

\begin{defn}
\label{defwt}
A weighted tree is a finite connected graph with no cycle, endowed with an edge weighted path metric.
\end{defn}

\begin{cor}
\label{lipwt}
If $T$ is a weighted tree, then $\fz(T)$ is isometric to $\dis {\ell}_1^{n}$ where $n=\text{card}(T)-1$.
\end{cor}

Embedding a finite metric space into a weighted tree leads to an easy computation of the norm of its Lipschitz-free space (this is not always possible as we shall see in the following section). Let us illustrate this remark with two obvious examples. A metric space  $M=\{0 \, , 1 \, , 2 \}$ of cardinal $3$ embeds into a four points weighted tree with edge lengths
$$
\lambda_0=\frac{1}{2} \  [ \, d_{01}+d_{02}-d_{12} \, ] \ , \
\lambda_1=\frac{1}{2} \ [ \, d_{01}+d_{12}-d_{02} \, ] \ , \
\lambda_2=\frac{1}{2} \ [ \, d_{02}+d_{12}-d_{01} \, ] \ .
$$
This embedding yields an isometry between $\fz(M)$ and a hyperplane of $\ell_1^3$ and we obtain
$$
\| \alpha_{1} \delta_{1} + \alpha_{2} \delta_{2} \| = \lambda_0 \, |\alp_1+\alp_2| + \lambda_1 \, | \alp_1|+\lambda_2 \, |\alp_2| \quad .
$$
If $M$ is an $(n+1)$-points metric space equipped with the discrete distance ($d(i \, , j)=1$ if $i \neq j$), we obtain in the same way
$$
\left\| \, \sum_{i=1}^{n} \alpha_{i} \delta_{i} \, \right\|
= \frac{1}{2} \, \left| \, \sum_{i=1}^{n} \alpha_{i} \, \right|
+\frac{1}{2} \, \sum_{i=1}^{n}|\alpha_{i}| \ \, .
$$
$\quad$
 
\section{Characterization of tree metrics}
In this section we show that subsets of $\rr$-trees are the only metric spaces with a Lipschitz-free space isometric to a subspace of an $\Lun$-space. We use the fact that a metric space isometrically embeds into an $\rr$-tree if and only if it is $0$-hyperbolic, that is to say satisfies the following property.

\begin{defn}
A metric space is said to satisfy the four-point condition if
$$d(a,b)+d(c,d) \, \leq \, \max \, (\, d(a,c)+d(b,d) \, , \, d(a,d)+d(b,c) \, )$$
whenever $a$, $b$, $c$ and $d$ are in $M$.
\end{defn}

Equivalently, the maximum of the three sums $d(a,b)+d(c,d)$, $d(a,c)+d(b,d)$ and $d(a,d)+d(b,c)$ is always attained at least twice. It is fairly straightfoward to check that metrics induced by $\rr$-trees indeed satisfy this condition; 
for a proof of the converse, see \cite{Eva2008}, Chap. 3.\\

The following theorem relates the aforementioned characterization of $\rr$-trees with Lipschitz-free spaces.

\begin{thm}
\label{main}
Let $M$ be a metric space. The following assertions are equivalent.
\textup{(1)} $\fz(M)$ is isometric to a subspace of an $L_1$-space.\\
\textup{(2)} $M$ satisfies the four-point condition.\\
\textup{(3)} $M$ isometrically embeds into an $\rr$-tree.
\end{thm}

\begin{proof}
\textup{(1)} $\Rightarrow$ \textup{(2)} Suppose $M$ is a metric space with a Lipschitz-free space isometric to a subspace of an $\Lun$-space, and consider four points in $M$ denoted $a_0$, $a_1$, $a_2$ and $a_3$. We will henceforth write $d_{ij}$ in place of $d(a_i \, , a_j)$. The unit ball of $\lipz(\{a_0 , a_1, a_2, a_3\})$ is isometric to the convex set of $\rr^3$ defined by the conditions
$$ \ |x| \leq d_{01}  \hspace{.2cm} ;  \hspace{.2cm} |y| \leq d_{02}  \hspace{.2cm} ;  \hspace{.2cm} |z| \leq d_{03}  \hspace{.2cm} ;  \hspace{.2cm} |y-x| \leq d_{12}
 \hspace{.2cm} ;  \hspace{.2cm} |z-x| \leq d_{13}  \hspace{.2cm} ;  \hspace{.2cm} |z-y| \leq d_{23}  \hspace{.2cm} .$$
Since $\fz(\{a_0 , a_1, a_2, a_3\})$ is isometric to a subspace of $\Lun$, this polytope is the projection of a cube and therefore all its faces admit a center of symmetry. The face obtained when making the third inequality to bind is characterized by
$$
\begin{array}{rcccl}
d_{03}-d_{13} & \leq & x & \leq & d_{01} \\
d_{03}-d_{23} & \leq & y & \leq & d_{02} \\
x-d_{12}   & \leq & y & \leq & x+d_{12}
\end{array}
$$
in the affine plane of equation $z = d_{03}$. \\
For any real numbers $a$, $b$, $c$, $d$, $e$, $f$, the planar convex set defined by
$$
\begin{array}{rcccl}
a & \leq & x & \leq & b \\
c & \leq & y & \leq & d \\
x+e   & \leq & y & \leq & x+f
\end{array}
$$
is either empty or centrally symmetric if and only if one of the following nine conditions holds.

$$
\begin{array}{ccccc}
b \ \leq \ a
&&
d \ \leq \ c
&&
f \ \leq \ e
\\
\\
d \ \leq \ a+e
&&
b+f \ \leq \ c
&&
a+b+e+f \ = \ c+d\\
\\
\left\lbrace
\begin{array}{rcl}
d & \leq & a+f \\
b+e & \leq & c
\end{array}
\right.
&&
\left\lbrace
\begin{array}{rcl}
b+f & \leq & d \\
c & \leq & a+e
\end{array}
\right.
&&
\left\lbrace
\begin{array}{rcl}
a+f & \leq & c\\
d & \leq & b+e
\end{array}
\right.
\end{array}
$$
Those different possibilities encompass segments, rectangles, parallelograms, and symmetric hexagons. Applying this to $a=d_{03}-d_{13}$, $b=d_{01}$, $c=d_{03}-d_{23}$, $d=d_{02}$, $e=-d_{12}$ and $f=d_{12}$, yields that one of the following conditions holds.

\begin{equation}
\label{cd1}
d_{03} \ \ = \ \ d_{01}+d_{13}
\end{equation}
\begin{equation}
\label{cd2}
d_{03} \ \ = \ \ d_{02}+d_{23}
\end{equation}
\begin{equation}
\label{cd3}
\left\lbrace
\begin{array}{rcl}
d_{01} + d_{23} & \leq &  d_{03}+d_{12}\\
d_{02}+d_{13} & \leq & d_{03}+d_{12}
\end{array}
\right.
\end{equation}
\begin{equation}
\label{cd4}
\left\lbrace
\begin{array}{rcl}
d_{02} & \ \ = \ \ & d_{01}+d_{12}\\
d_{23} & \ \ = \ \ & d_{12}+d_{13}
\end{array}
\right.
\end{equation}
\begin{equation}
\label{cd5}
\left\lbrace
\begin{array}{rcl}
d_{01} & \ \ = \ \ & d_{02}+d_{12}\\
d_{13} & \ \ = \ \ & d_{12}+d_{23}
\end{array}
\right.
\end{equation}
\begin{equation}
\label{cd6}
d_{01}+d_{23} \ \ = \ \ d_{02}+d_{13}
\end{equation}
\\
Both (\ref{cd1}) and (\ref{cd2}) imply (\ref{cd3}). Also, each of (\ref{cd4}) and (\ref{cd5}) imply (\ref{cd6}), which leaves us with

\begin{equation}
\left\lbrace
\begin{array}{rcl}
d_{01} + d_{23} & \leq &  d_{03}+d_{12}\\
d_{02}+d_{13} & \leq & d_{03}+d_{12}
\end{array}
\right.
\nonumber
\end{equation}

\begin{equation}
d_{01}+d_{23} \ \ = \ \ d_{02}+d_{13}
\nonumber
\end{equation}
\\
either condition implying $d_{01}+d_{23} \leq  \max \ ( \, d_{02}+d_{13} \, , \, d_{03} + d_{12} \,)$.\\

\noi \textup{(2)} $\Rightarrow$ \textup{(3)} see \cite{Eva2008}.\\

\noi
\textup{(3)} $\Rightarrow$ \textup{(1)}  results from corollary \ref{lipt}.
\end{proof}

In the special case when $M$ is separable, the $L_1$-space in condition (1) can be taken as $L_1$, any separable isometric subspace of an $L_1$-space being an isometric subspace of $L_1$. In case $M$ is finite, it can be replaced by $\ell_1^N$ where $N=2 \times \text{Card} M-3$ (it is shown in \cite{Bun1974} that a finite metric space with the four-point property isometrically embeds into a weighted tree with $2 \times \text{Card} M -2$ vertices).\\

Note that a polyhedral space isometrically embeds into $L_1$ if and only if its 3-dimensional subspaces do so (the so-called Hlawka's inequality which characterizes this property involves only three vectors); as far as Lipschitz-free spaces $\fz(M)$ are concerned, embeddability of the 3-dimensional subspaces corresponding to Lipschitz-free spaces of  4-points subsets of $M$ is enough, and we obtain a canonical embedding.

\section{\lipfss isomorphic to a subspace of $\Lun$}
We show in this section a very natural gluing result: if $M$ is a metric space of finite radius composed of uniformly apart subsets $M_{\gamma}$'s, $\fz(M)$ is essentially the $\ell_1$-sum of the Lipschitz-free spaces $\fz(M_\gamma)$. 

\begin{prop}
\label{isomorphism}
Let $\Gamma$ be a pointed set and $M=\dis \cup_{\gamma \in \Gamma} M_{\gamma}$ a metric space. If there exists positive constants $\alpha$ and $\beta$ such that $\alpha \leq d(x,y) \leq \beta$ whenever $x$ and $y$ belong to distincts $M_{\gamma}$'s, then we have the isomorphism
$$\fz(M) \simeq \left( \und{\gamma \in \Gamma}{\sum} \ \fz(M_\gamma) \right)_{\lun} \oplus_{1} \lun(\Gamma^{*}) \ .$$
\end{prop}

The  $\lun$-space compensates for the loss of dimensions due to the choice of a base point in each $M_{\gamma}$: each connected component of $M$ contributes for a dimension in this $\ell_1$-space.

\begin{proof}
For each $\gamma \in \Gamma$, we choose a base point $a_\gamma$ in $M_{\gamma}$ in order to define $\fz(M_\gamma)$, and use $a_0$ to define $\fz(M)$. We consider the following two maps.

$$
\Phi :
\begin{cases}
\begin{array}{ccl}
\left(\und{\gamma \in \Gamma}{\sum} \text{Lip}_0 (M_{\gamma}) \right)_{\ell_{\infty}} \oplus_{\infty}
\ell_{\infty}(\Gamma^*)
 & \to & \text{Lip}_0 (M)\\
((f_{\gamma})_{\gamma \in \Gamma} \, , (\lambda_{\gamma})_{\gamma \in \Gamma^*}) & \mapsto &
f =
\begin{cases}
\begin{array}{cc}
f_0 & \text{on} \ M_0 \\
f_\gamma + \lambda_{\gamma} & \text{on} \ M_\gamma
\end{array}
\end{cases}
\end{array}
\end{cases}
$$
\\
and

$$
\Psi :
\begin{cases}
\begin{array}{ccl}
\text{Lip}_0 (M)
 & \to & \left(\und{\gamma \in \Gamma}{\sum} \text{Lip}_0 (M_{\gamma}) \right)_{\ell_{\infty}} \oplus_{\infty}
\ell_{\infty}(\Gamma^*)
\\
f & \mapsto &
((f|_{\gamma}-f(a_\gamma))_{\gamma \in \Gamma} \, , (f(a_{\gamma}))_{\gamma \in \Gamma^*})
\end{array}
\end{cases}
$$
\\

First of all let us show that $\Phi$ is well-defined, and has norm less than 
$2 (\alpha+\beta+1)/\alpha$.
If $u=((f_{\gamma})_{\gamma \in \Gamma} \, , (\lambda_{\gamma})_{\gamma \in \Gamma^*})$, we have for any $x$ in $M_{\sigma}$ and $y$ in $M_{\tau}$,
$$|\Phi(u)(y)-\Phi(u)(x)| \, \leq \, |f_{\sigma}(x)|+|f_{\tau}(y)|+|\lambda_{\sigma}|+|\lambda_{\tau}|$$
\noi (with $\lambda_{0}=0$).  Now, assuming  $\sigma \neq \tau$, we get
$$
|f_{\sigma}(x)| \, \leq \, \|f_{\sigma}\| \, d(x,a_{\sigma}) \, \leq \, \|f_{\sigma}\| \, (d(x,y)+d(y,a_{\sigma}))
\, \leq \, \|u\| \, (d(x,y)+\beta) \quad .
$$

\noi The inequality $1 \, \leq \, \dfrac{1}{\alpha} \ d(x,y)$ yields

$$
|f_{\sigma}(x)| \, \leq \, \left(1+\frac{\beta}{\alpha} \right) \, \|u\| \, d(x,y)
\quad \text{and} \quad
|\lambda_{\sigma}| \, \leq \,   \frac{1}{\alpha} \, \|u\| \, d(x,y)
$$
these upper bounds being also valid for $\tau$. Hence
$$|\Phi(u)(y)-\Phi(u)(x)| \, \leq \, 2 \ \frac{\alpha+\beta+1}{\alpha} \ \|u\| \ d(x,y) \quad .$$

The latter inequality holds if $x$ and $y$ belong to the same $M_{\sigma}$, and we conclude that $\Phi(u)$ is a Lipschitz map of norm less than $2 \ \frac{\alpha+\beta+1}{\alpha} \ \|u\|$. The map $\Psi$ is the inverse of $\Phi$, and we have $\| \Psi \| \leq\max(1,\beta)$. The isomorphism $\Phi$ being weak*-weak* continuous, it is the transpose of a linear isomorphism between $\fz(M)$ and $\left(\oplus_{\ell_1} \fz(M_\gamma) \right) \oplus_{1} \lun(\Gamma^{*})$.
\end{proof}

Combined with the results of the preceding section, this gives canonical examples of Lipschitz-free spaces isomorphic to a subspace of $\Lun$ (and therefore unique isometric predual by \cite{GodTal1981}).

Examples of spaces $\fz(M)$ which are unique isometric preduals are plentiful in \cite{Kal2004}; N.J. Kalton shows the following result: if $M$ is an arbitrary metric space and $\eps >0$, then the space $\fz(M)$ is $(1+\eps)$-isometric to a subspace of 
$\left(\sum_{k \in \zz} \ \fz(M_k) \right)_{\lun}$
where $M_k=\left\{x \in M \, / \, d(x,0) \leq 2^k \hs \right\}$. It is then deduced that when $M$ is uniformly discrete (i.e. $\inf_{x \neq y} d(x \, , y) >0$), $\fz(M)$ is a Schur space with the Radon-Nikodym Property and the approximation property. In this case, $\fz(M)$ is unique isometric predual since it has the  Radon-Nikodym Property (see \cite{God1989} p.144). Also, if $K$ is a compact subset of a finite-dimensional normed space and $0<\alpha<1$, the metric space $K$ equipped with the distance $d(x,y)=\|y-x\|^\alpha$ has a Lipschitz-free space (denoted $\fz_\alpha(K)$) isomorphic to $\lun$: hence $\fz_\alpha(A)$ is unique isometric predual whenever $A$ is a subset of a finite-dimensional normed space. Note that Kalton shows that if $K$ is a compact convex subset of $\ell_2$ and is infinite-dimensional then $\fz_\alpha(K)$ cannot be isomorphic to $\ell_1$.

\nocite{GooWei1993}
\nocite{Rud1987}

\bibliographystyle{plain}
\bibliography{lip}

\subsection*{Acknowledgments}$\quad$\\
The author thanks Guillaume Aubrun for having introduced him to the theory of zonoids, and Gilles Lancien for useful discussions and comments.

\end{document}